\documentclass[reqno]{amsart}

\usepackage{amssymb,amsmath,amsthm,xcolor,enumerate,hyperref,cleveref}
\usepackage[framemethod=tikz]{mdframed}

\allowdisplaybreaks

\newtheorem{thrm}{Theorem}[section]

\newtheorem{lem}[thrm]{Lemma}

\theoremstyle{definition}
\newtheorem{defn}[thrm]{Definition}

\newtheorem{rem}[thrm]{Remark}

\crefrangeformat{equation}{#3(#1)#4--#5(#2)#6}

\crefname{thrm}{Theorem}{Theorems}
\crefname{lem}{Lemma}{Lemmas}
\crefname{cor}{Corollary}{Corollaries}
\crefname{prop}{Proposition}{Propositions}
\crefname{defn}{Definition}{Definitions}
\crefname{exm}{Example}{Examples}
\crefname{rem}{Remark}{Remarks}
\crefname{section}{Section}{Sections}
\crefname{equation}{\unskip}{\unskip}
\crefname{enumi}{\unskip}{\unskip}

\renewcommand{\iff}{\Leftrightarrow}

\newcommand{\cR}{\mathcal R}

\begin{document}

\title[Lie Derivations of Incidence Algebras]{Lie Derivations of Incidence Algebras}

\author{Xian Zhang}
\address{School of Mathematical Sciences, Huaqiao University,
Quanzhou, Fujian, 362021, P. R. China}
\email{xianzhangmath@163.com}

\author{Mykola Khrypchenko}
\address{Departamento de Matem\'atica, Universidade Federal de Santa Catarina,  Campus Reitor Jo\~ao David Ferreira Lima, Florian\'opolis, SC, CEP: 88040--900, Brazil}
\email{nskhripchenko@gmail.com}

\begin{abstract}
Let $X$ be a locally finite preordered set, $\cR$ a commutative ring with identity and $I(X,\cR)$ the incidence algebra of $X$ over $\cR$. In this note we prove that each Lie derivation of $I(X,\cR)$ is proper, provided that $\cR$ is $2$-torsion free.
\end{abstract}



\maketitle

\section{Introduction and Preliminaries}\label{intro-prelim}

Let $A$ be an associative algebra over a commutative ring $\cR$ (possibly without identity). We define the {\it Lie product} $[x,y]:=xy-yx$ and {\it Jordan product} $x\circ y:=xy+yx$ for all $x,y\in A$. Then $(A,[\ ,\ ])$ is a Lie algebra and $(A,\circ)$ is a Jordan algebra. It is fascinating to study the connection between the associative, Lie and Jordan structures on $A$. In this field, two classes of maps are of crucial importance. One of them consists of maps, preserving a type of product, for example, Lie homomorphisms etc. The other one is formed by differential operators, satisfying a type of Leibniz formulas, for example, Jordan derivations etc.

We recall that an $\cR$-linear map $D: A\to A$ is called a {\it derivation} if $D(xy)=D(x)y+xD(y)$ for all $x,y\in A$, and it is called a {\it Lie derivation} if
$$
D([x,y])=[D(x),y]+[x,D(y)]
$$
for all $x,y\in A$. Note that if $D$ is a derivation of $A$ and $F$ is an $\cR$-linear map from $A$ into its center, then $D+F$ is a Lie derivation if and only if $F$ annihilates all commutators $[x,y]$. A Lie derivation of the form $D+F$, with $D$ being a derivation and $F$ a central-valued map, will be called {\it proper}. In this article we find a class of algebras on which every Lie derivation is proper.

In the AMS Hour Talk of 1961 Herstein proposed many problems concerning the structure of Jordan and Lie maps in associative simple and prime rings~\cite{Her}. Roughly speaking, he conjectured that these maps are all of the proper or standard form. The renowned Herstein's Lie-type mapping research program was formulated since then. Martindale gave a major force in this program under the assumption that the rings contain some nontrivial idempotents, see~\cite{Mar64} for example. The first idempotent-free result on Lie-type maps was obtained by Bre\v{s}ar in~\cite{Bre93}. We refer the reader to Bre\v{s}ar's survey paper~\cite{Bre04} for a much more detailed historic background.

Let us now recall another notion, incidence algebra~\cite{Kopp,SpDo}, with which we deal in this paper. Let $(X,\le)$ be a locally finite preordered set. This means that $\le$ is a reflexive and transitive binary relation on $X$, and for any $x\le y$ in $X$ there are only finitely many elements $z$ satisfying $x\le z\le y$. Given a commutative ring with identity $\cR$, the {\it incidence algebra} $I(X,\cR)$ of $X$ over $\cR$ is defined to be the set
$$
I(X,\cR):=\{f: X\times X\to \cR\mid f(x,y)=0\ \mbox{if}\ x\not\le y\}
$$
with algebraic operations given by
\begin{align*}
(f+g)(x,y)&=f(x,y)+g(x,y),\\
(rf)(x,y)&=rf(x,y),\\
(fg)(x,y)&=\sum_{x\le z\le y}f(x,z)g(z,y)
\end{align*}
for all $f,g\in I(X,\cR)$, $r\in \cR$ and $x,y\in X$. The product $fg$ is usually called the {\it convolution} in function theory. The identity element $\delta$ of $I(X,\cR)$ is given by $\delta(x,y)=\delta_{xy}$ for $x\le y$, where $\delta_{xy}\in \{0,1\}$ is the Kronecker delta. It is clear that the full matrix algebra ${\rm M}_n(\cR)$ and the upper triangular matrix algebra ${\rm T}_n(\cR)$ are special examples of incidence algebras.

The incidence algebra of a partially ordered set (poset) $X$ was first considered by Ward in~\cite{Wa} as a generalized algebra of arithmetic functions. Rota and Stanley developed incidence algebras as fundamental structures of enumerative combinatorial
theory and the allied areas of arithmetic function theory (see~\cite{Stanley}). Furthermore, Stanley~\cite{St} initiated the study of algebraic maps and combinatorial structure of an incidence algebra. Since then the automorphisms and other algebraic maps of incidence algebras have been extensively studied (see \cite{BruL,BruFS,Kh,Kh-aut,Kh-der,Kopp,Sp} and the references therein). On the other hand,
in the theory of operator algebras, the incidence algebra $I(X,\cR)$ of a finite poset $X$ is referred as a bigraph algebra or a finite dimensional CSL algebra. Therefore, one of the main motivations of this paper is to connect the Herstein's program to operator algebras in a combinatorial flavor.

The connection between the Herstein's program and operator algebras in analysis field has been studying for several decades. The operator algebras on which every Lie derivation is proper include von Neumann algebras~\cite{Miers}, certain CSL algebras~\cite{Lu}, nest algebras~\cite{Cheung}, $C^{\ast}$-algebras~\cite{MaV}, etc. However, Cheung's work~\cite{Cheung} is of special significance
in our present case. He viewed the nest algebra over a Hilbert space as a triangular algebra and hence avoided the analysis technique. This method makes it possible to study Lie derivations on incidence algebras in a combinatorial and linear manner (see also \cite{Ben-LieDer-TM,DuWang-LieDer-GMA,Ben-GenLieDer-TA,XiaoWei-LieTriple-TA}).

The content of this article is organized as follows. In \cref{fin-case} we study Lie derivations of the incidence algebra $I(X,\cR)$, when $X$ is a finite preordered set and $\cR$ is a $2$-torsion free commutative ring with identity. The case of $X$ being a locally finite preordered set will be considered in \cref{gen-case} by extending a Lie derivation of a subalgebra to the whole algebra $I(X,\cR)$.

\section{The finite case}\label{fin-case}

Throughout this article $\cR$ will always be a $2$-torsion free commutative ring with identity. We begin with the case, when $X$ is finite preordered set.

For each pair $x\le y$ define $e_{xy}$ by
\begin{align}\label{e_xy}
e_{xy}(u,v)=
 \begin{cases}
  1, & \mbox{if }(u,v)=(x,y),\\
  0, & \mbox{otherwise}.
 \end{cases}
\end{align}
Then $e_{xy}e_{uv}=\delta_{yu}e_{xv}$ by the definition of convolution. Moreover, the set $\mathfrak{B}:=\{e_{xy}\mid x\le y\}$ forms an $\cR$-linear basis of $I(X,\cR)$, which will be called {\it standard}.
The main result of this section is as follows.

\begin{thrm}\label{L-proper-fin-case}
If $X$ is finite, then each Lie derivation of $I(X,\cR)$ is proper.
\end{thrm}

Observe that it is enough to prove \cref{L-proper-fin-case} for a connected $X$. Indeed, let $X=\bigsqcup_{i\in I}X_i$ be the decomposition of $X$ into distinct connected components, where $I$ is a finite index set. Denote by $\delta_i$ the element $\sum_{x\in X_i} e_{xx}\in I(X,\cR)$. It follows from~\cite[Theorem  1.3.13]{SpDo} that $\{\delta_i\mid i\in I\}$ forms
a complete set of central primitive idempotents of $I(X,\cR)$. In other words, $I(X,\cR)=\bigoplus_{i\in I}\delta_i
I(X,\cR)$. It is clear that $\delta_i I(X,\cR)\cong I(X_i,\cR)$ for each $i\in I$. Hence we only need to prove \cref{L-proper-fin-case} when $X$ is connected by \cite[Proposition 2]{Cheung}.

The following lemma, whose proof is standard (see, for example, \cite[Lemma 2.1]{Xiao1}), enables us to consider only the action of a Lie derivation on the basis $\mathfrak{B}$.

\begin{lem}\label{L-on-basis}
Let $A$ be an $\cR$-algebra with an $\cR$-linear basis $Y$. Then an $\cR$-linear operator
$L: A \to A$ is a Lie derivation if and only if
$$
L\left([x,y]\right)=[L(x),y]+[x,L(y)]
$$
for all  $x,y\in Y$.
\end{lem}


Let $L:I(X,\cR)\to I(X,\cR)$ be an $\cR$-map. Motivated by the above lemma, for all $i\le j$ we denote by $C_{xy}^{ij}$ the coordinates of $L(e_{ij})$ in the basis $\mathfrak B$, namely,
$$
L(e_{ij})=\sum_{x\le y}C_{xy}^{ij}e_{xy}.
$$
We make the convention $C_{xy}^{ij}=0$, if needed, for $x\not\le y$. We shall also use the standard notation $i<j$ or $j>i$, if $i\le j$ and $i\ne j$.


\begin{lem}\label{L(e_ij)-gen-form}
Let $X$ be connected and $L$ be a Lie derivation of $I(X,\cR)$. Then
\begin{align}
L(e_{ii})&=\sum_{x<i}C_{xi}^{ii}e_{xi}+\sum_{x\in X}C_{xx}^{ii}e_{xx}+\sum_{y>i}C_{iy}^{ii}e_{iy},\label{L(e_ii)-final}\\
L(e_{ij})&=\sum_{x<i}C_{xi}^{ii}e_{xj}+C_{ij}^{ij}e_{ij}+\sum_{y>j}C_{jy}^{jj}e_{iy},\ \mbox{if}\ i<j.\label{L(e_ij)-final}
\end{align}
\end{lem}
\begin{proof}
Suppose first that $|X|=1$. Then \cref{L(e_ii)-final} takes the trivial form $L(e_{ii})=C_{ii}^{ii}e_{ii}$, where $i$ is the unique element of $X$. Since there is no pair of $i<j$ in $X$, formula \cref{L(e_ij)-final} also trivially holds.

Let $|X|\ge 2$ and $i\in X$. Since $X$ is connected, there is an element $j\ne i$ comparable with $i$. Consider first the case $i<j$.
As $e_{ij}=[e_{ii},e_{ij}]$, we have
\begin{align}\label{L(e_ij)=[L(e_ii)e_ij]+[e_iiL(e_ij)]}
L(e_{ij})=[L(e_{ii}),e_{ij}]+[e_{ii},L(e_{ij})]=L(e_{ii})e_{ij}-e_{ij}L(e_{ii})+e_{ii}L(e_{ij})-L(e_{ij})e_{ii}.
\end{align}
Multiplying \cref{L(e_ij)=[L(e_ii)e_ij]+[e_iiL(e_ij)]} by $e_{ii}$ on the left and by $e_{yy}$ on the right, we get
\begin{align}
C^{ii}_{ii}&=C^{ii}_{jj},\label{C^ii_ii=C^ii_jj}\\
C^{ii}_{jy}&=0,\ \mbox{if} \ y\ne i,j.\label{C^ii_jy=0}
\end{align}
Applying the above relations \cref{C^ii_ii=C^ii_jj,C^ii_jy=0} we obtain
\begin{align}
L(e_{ii})&=\sum_{x\le y}C_{xy}^{ii}e_{xy}=\sum_{j\ne x\le y}C_{xy}^{ii}e_{xy}+C_{ji}^{ii}e_{ji}+C_{jj}^{ii}e_{jj}\notag\\
&=\sum_{j\ne x<i}C_{xi}^{ii}e_{xi}+C_{ii}^{ii}e_{ii}
 +\sum_{j\ne x\le y\ne i}C_{xy}^{ii}e_{xy}+C_{ji}^{ii}e_{ji}+C_{jj}^{ii}e_{jj}\notag\\
&=\sum_{x<i}C_{xi}^{ii}e_{xi}+C_{ii}^{ii}(e_{ii}+e_{jj})+\sum_{y>i}C_{iy}^{ii}e_{iy}
 +\sum_{i,j\ne x\le y\ne i}C_{xy}^{ii}e_{xy}.\label{L(e_ii)-not-simplified}
\end{align}
On the other hand, $L\left([e_{ii},e_{xx}]\right)=0$ for any $x\in X$. Thus
$$
L(e_{ii})e_{xx}-e_{xx}L(e_{ii})+e_{ii}L(e_{xx})-L(e_{xx})e_{ii}=0.
$$
Left multiplication by $e_{xx}$ and right multiplication by $e_{yy}$ in the above equation leads to
$$
C^{ii}_{xy}=0,\ \mbox{if}\ i\ne x<y\ne i.
$$
Therefore the equation \cref{L(e_ii)-not-simplified} can be rewritten as
$$
L(e_{ii})=\sum_{x<i}C_{xi}^{ii}e_{xi}+C_{ii}^{ii}(e_{ii}+e_{jj})
 +\sum_{y>i}C_{iy}^{ii}e_{iy}+\sum_{x\neq i,j}C_{xx}^{ii}e_{xx},
$$
giving \cref{L(e_ii)-final}.

Suppose now that $j<i$. Then in the dual poset $(X^{op},\le^{op})$ (see \cite[p. 2]{SpDo}) one has $i<^{op}j$. Observe that $I(X^{op},\cR)$ is anti-isomorphic to $I(X,\cR)$, the anti-isomorphism being the map $\varphi(e^{op}_{ij})=e_{ji}$, where $e^{op}_{ij}$ is an element of the standard basis of $I(X^{op},\cR)$. Then $L^{op}(e^{op}_{ij})=\varphi^{-1}(L(e_{ji}))$ defines a Lie derivation on $I(X^{op},\cR)$, more precisely $L^{op}(e^{op}_{ij})=\sum_{x\le y}C^{ji}_{xy}e^{op}_{yx}=\sum_{x\le^{op} y}C^{ji}_{yx}e^{op}_{xy}$. Applying the result of the previous case to $L^{op}$, we obtain
$$
 L^{op}(e^{op}_{ii})=\sum_{x<^{op}i}C_{ix}^{ii}e^{op}_{xi}+\sum_{x\in X}C_{xx}^{ii}e^{op}_{xx}+\sum_{y>^{op}i}C_{yi}^{ii}e^{op}_{iy}.
$$
Then
$$
L(e_{ii})=\varphi(L^{op}(e^{op}_{ii}))=\sum_{x<^{op}i}C_{ix}^{ii}e_{ix}+\sum_{x\in X}C_{xx}^{ii}e_{xx}+\sum_{y>^{op}i}C_{yi}^{ii}e_{yi}.
$$
Changing $x$ by $y$ in the first sum, $y$ by $x$ in the third sum, and observing that $x<^{op}i\iff x>i$, and $y>^{op}i\iff y<i$, we obtain \cref{L(e_ii)-final}.

We next describe the form of $L(e_{ij})$ for any $i<j$. Observe first that left multiplication by $e_{xx}$ and right multiplication by $e_{jj}$ in $L(e_{ij})=L\left([e_{ij},e_{jj}]\right)$ leads to
\begin{align}
C^{jj}_{jj}&=C^{jj}_{ii},\label{C^jj_jj=C^jj_ii}\\
C^{jj}_{xi}&=0,\  \mbox{if} \ x\neq i,j.\label{C_xi^jj=0}
\end{align}
Furthermore, $e_{ij}=[[e_{ii},e_{ij}],e_{jj}]$. Then it follows from \cref{C^ii_ii=C^ii_jj,C^ii_jy=0,L(e_ii)-final,C^jj_jj=C^jj_ii,C_xi^jj=0} that
\begin{align}
L(e_{ij})&=[[L(e_{ii}),e_{ij}],e_{jj}]+[[e_{ii},L(e_{ij})],e_{jj}]+[e_{ij},L(e_{jj})]\notag\\
&=[[\sum_{x<i}C_{xi}^{ii}e_{xi}+\sum_{x\in X}C_{xx}^{ii}e_{xx}+\sum_{y>i}C_{iy}^{ii}e_{iy},e_{ij}],e_{jj}]\notag\\
&\quad +[[e_{ii},\sum_{x\le y}C_{xy}^{ij}e_{xy}],e_{jj}]\notag\\
&\quad +[e_{ij},\sum_{x<j}C_{xj}^{jj}e_{xj}+\sum_{x\in X}C_{xx}^{jj}e_{xx}+\sum_{y>j}C_{jy}^{jj}e_{jy}]\notag\\
&=\sum_{x<i}C_{xi}^{ii}e_{xj}+C_{ij}^{ij}e_{ij}-\left(C_{ji}^{ii}+C_{ji}^{jj}\right)e_{jj}+C_{ji}^{ij}e_{ji}+\sum_{y>j}C_{jy}^{jj}e_{iy}.\label{L([[e_iie_ij]e_jj])}
\end{align}
Analogously from $e_{ij}=[e_{ii},[e_{ij},e_{jj}]]$ we deduce
\begin{align}
L(e_{ij})=\sum_{x<i}C_{xi}^{ii}e_{xj}+C_{ij}^{ij}e_{ij}
 -\left(C_{ji}^{ii}+C_{ji}^{jj}\right)e_{ii}+C_{ji}^{ij}e_{ji}
 +\sum_{y>j}C_{jy}^{jj}e_{iy}.\label{L([e_ii[e_ije_jj]])}
\end{align}
Combining \cref{L([[e_iie_ij]e_jj])} with \cref{L([e_ii[e_ije_jj]])} yields that
\begin{align}\label{C_ji^ii+C_ji^jj=0}
 C_{ji}^{ii}+C_{ji}^{jj}=0,
\end{align}
since $i\neq j$. On the other hand, left multiplication by $e_{jj}$ and right multiplication by $e_{ii}$ in \cref{L(e_ij)=[L(e_ii)e_ij]+[e_iiL(e_ij)]} leads to $C^{ij}_{ji}=-C^{ij}_{ji}$, which implies that $C^{ij}_{ji}=0$, since $\cR$ is $2$-torsion free. Therefore, we get the desired form \cref{{L(e_ij)-final}}.
\end{proof}

\begin{lem}\label{coefficients-C^ij_xy}
Let $X$ be connected. An $\cR$-linear map $L: I(X,\cR)\to I(X,\cR)$ of the form \cref{L(e_ii)-final,L(e_ij)-final} is a Lie derivation of $I(X,\cR)$ if and only if the coefficients $C_{xy}^{ij}$ are subject to the following relations
\begin{align}
C_{ij}^{ii}+C_{ij}^{jj}&=0, \ \mbox{if}\ i<j,\label{C_ij^ii+C_ij^jj=0}\\
C_{ij}^{ij}+C_{jk}^{jk}&=C_{ik}^{ik}, \  \mbox{if}\ i<j<k\ \mbox{and}\ i\ne k,\label{C_ij^ij+C_jk^jk=C_ik^ik}\\
C_{ij}^{ij}+C_{ji}^{ji}&=0, \  \mbox{if}\ i<j<i,\label{C_ij^ij+C_ji^ji=0}\\
C_{ii}^{ii}&=C_{xx}^{ii}, \ \mbox{for all}\ x\in X.\label{C_ii^ii=C_xx^ii}
\end{align}
\end{lem}
\begin{proof}
By \cref{L-on-basis} in order to determine that $L$ is a Lie derivation, it is necessary and sufficient to show that $L([e_{ij},e_{kl}])=[L(e_{ij}),e_{kl}]+[e_{ij},L(e_{kl})]$ for all $i\le j$ and $k\le l$. Since $[e_{ij},e_{kl}]=\delta_{jk}e_{il}-\delta_{li}e_{kj}$, there are two cases occurring.

{\bf Case 1.} If $i=j$, then one has two subcases.

{\it Case 1.1.} When $k=l$, we have by \cref{L(e_ii)-final}
\begin{align}
0&=L\left([e_{ii},e_{kk}]\right)\notag\\
&=L(e_{ii})e_{kk}-e_{kk}L(e_{ii})+e_{ii}L(e_{kk})-L(e_{kk})e_{ii}\notag\\
&=\delta_{ik}\sum_{x<i}C_{xi}^{ii}e_{xk}-\delta_{ik}\sum_{y>i}C_{iy}^{ii}e_{ky}
 +\delta_{ik}\sum_{y>k}C_{ky}^{kk}e_{iy}\notag\\
&\quad -\delta_{ik}\sum_{x<k}C_{xk}^{kk}e_{xi}+\left(C_{ik}^{ii}+C_{ik}^{kk}\right)e_{ik}-\left(C_{ki}^{ii}+C_{ki}^{kk}\right)e_{ki}.\label{0=L([e_iie_kk])}
\end{align}
If $i\ne k$ and $i,k$ are comparable, then \cref{0=L([e_iie_kk])} is equivalent to $C_{ik}^{ii}+C_{ik}^{kk}=0$ for $i<k$ and $C_{ki}^{ii}+C_{ki}^{kk}=0$ for $k<i$, so we get \cref{{C_ij^ii+C_ij^jj=0}}. If $i=k$ or $i,k$ are incomparable, then \cref{0=L([e_iie_kk])} always holds.

{\it Case 1.2.} When $k\neq l$, we have
\begin{align}
0&=L\left([e_{ii},e_{kl}]\right)-\delta_{ik}L(e_{il})+\delta_{il}L(e_{ki})\notag\\
&=L(e_{ii})e_{kl}-e_{kl}L(e_{ii})+e_{ii}L(e_{kl})-L(e_{kl})e_{ii}\notag\\
&\quad -\delta_{ik}L(e_{il})+\delta_{il}L(e_{ki}).\label{0=L([e_iie_kl])-d_ikL(e_il)+d_ilL(e_ki)}
\end{align}
If $k\ne i\ne l$, then by \cref{L(e_ii)-final,L(e_ij)-final} equality \cref{0=L([e_iie_kl])-d_ikL(e_il)+d_ilL(e_ki)} can be rewritten as
\begin{align*}
0&=L(e_{ii})e_{kl}-e_{kl}L(e_{ii})+e_{ii}L(e_{kl})-L(e_{kl})e_{ii}\\
&=\left(C_{kk}^{ii}-C_{ll}^{ii}\right)e_{kl}+\left(C_{ik}^{ii}+C_{ik}^{kk}\right)e_{il}-\left(C_{li}^{ii}+C_{li}^{ll}\right)e_{ki}.
\end{align*}
Hence,
\begin{align}\label{C_kk^ii=C_ll^ii-for-k-ne-i-ne-l}
 C_{kk}^{ii}=C_{ll}^{ii}\mbox{ for } k<l\mbox{ and }k\neq i\neq l,
\end{align}
$C_{ik}^{ii}+C_{ik}^{kk}=0$ for $i<k$, $C_{li}^{ii}+C_{li}^{ll}=0$ for $l<i$. If $k\ne i=l$, then \cref{0=L([e_iie_kl])-d_ikL(e_il)+d_ilL(e_ki)} can be rewritten as
\begin{align*}
0&=L(e_{ii})e_{kl}-e_{kl}L(e_{ii})+e_{ii}L(e_{kl})-L(e_{kl})e_{ii}+L(e_{ki})\\
&=\left(C_{kk}^{ii}-C_{ii}^{ii}\right)e_{ki}+\left(C_{ik}^{ii}+C_{ik}^{kk}\right)e_{ii}.
\end{align*}
This gives
\begin{align}\label{C_kk^ii=C_ii^ii-for-k<i}
 C_{kk}^{ii}=C_{ii}^{ii}\mbox{ for }k<i,
\end{align}
and $C_{ik}^{ii}+C_{ik}^{kk}=0$ for $i<k<i$. If $k=i\ne l$, then, similarly, we have
\begin{align}\label{C_ii^ii=C_ll^ii-for-i<l}
 C_{ii}^{ii}=C_{ll}^{ii}\mbox{ for }i<l,
\end{align}
and $C_{li}^{ii}+C_{li}^{ll}=0$, for $l<i<l$. Combining \cref{C_kk^ii=C_ll^ii-for-k-ne-i-ne-l,C_kk^ii=C_ii^ii-for-k<i,C_ii^ii=C_ll^ii-for-i<l}, we get
\begin{align}\label{C^ii_xx=C^ii_yy-for-x-le-y}
 C^{ii}_{xx}=C^{ii}_{yy} \mbox{ for all }x\le y.
\end{align}
Since $X$ is connected, there is a path from $i$ to any $x\in X$. Using \cref{C^ii_xx=C^ii_yy-for-x-le-y} and induction on the length of the path, we obtain \cref{C_ii^ii=C_xx^ii}.

{\bf Case 2.} If $i\neq j$, we can also assume that $k\neq l$ (the case $k=l$ is symmetric to {\it Case 1.2}) and
there are four subcases appearing.

{\it Case 2.1.} When $i\neq l$ and $j\neq k$, using \cref{L(e_ii)-final,L(e_ij)-final} we have
\begin{align}
0&=[L(e_{ij}),e_{kl}]+[e_{ij},L(e_{kl})]\notag\\
&=\left(C_{jk}^{jj}+C_{jk}^{kk}\right)e_{il}-\left(C_{li}^{ii}+C_{li}^{ll}\right)e_{kj}.\label{0=[L(e_ij)e_kl]+[e_ijL(e_kl)]}
\end{align}
If $i\ne k$ or $j\ne l$, we get the relations $C_{jk}^{jj}+C_{jk}^{kk}=0$ for $j<k$, and $C_{li}^{ii}+C_{li}^{ll}=0$ for $l<i$, which are \cref{C_ij^ii+C_ij^jj=0}. If $i=k$ and $j=l$, then \cref{0=[L(e_ij)e_kl]+[e_ijL(e_kl)]} always holds.

{\it Case 2.2.} When $i=l$ and $j\neq k$, we have by \cref{L(e_ij)-final}
\begin{align}
-L(e_{kj})&=[L(e_{ij}),e_{ki}]+[e_{ij},L(e_{ki})]\notag\\
&=C_{jk}^{jj}e_{ii}-C_{ij}^{ij}e_{kj}-\sum_{y>j}C_{jy}^{jj}e_{ky}\notag\\
&\quad +C_{jk}^{kk}e_{ii}-\sum_{x<k}C_{xk}^{kk}e_{xj}-C_{ki}^{ki}e_{kj}\notag\\
&=-\sum_{x<k}C_{xk}^{kk}e_{xj}-\left(C_{ij}^{ij}+C_{ki}^{ki}\right)e_{kj}-\sum_{y>j}C_{jy}^{jj}e_{ky}+\left(C_{jk}^{jj}+C_{jk}^{kk}\right)e_{ii}.\label{-L(e_kj)=[L(e_ij)e_ki]+[e_ijL(e_ki)]}
\end{align}
Comparing \cref{-L(e_kj)=[L(e_ij)e_ki]+[e_ijL(e_ki)]} with the form of $-L(e_{kj})$ from \cref{L(e_ij)-final}, we obtain $C_{jk}^{jj}+C_{jk}^{kk}=0$ for $j<k$, and $C_{ki}^{ki}+C_{ij}^{ij}=C_{kj}^{kj}$ for $k<i<j$, $k\ne j$, the latter being \cref{C_ij^ij+C_jk^jk=C_ik^ik}.

{\it Case 2.3.} The case, when $i\neq l$ and $j=k$, is symmetric to {\it Case 2.2}.

{\it Case 2.4.} When $i=l$ and $j=k$, we have by \cref{L(e_ij)-final}
\begin{align}
L(e_{ii})-L(e_{jj})&=[L(e_{ij}),e_{ji}]+[e_{ij},L(e_{ji})]\notag\\
&=\sum_{x<i}C_{xi}^{ii}e_{xi}+C_{ij}^{ij}e_{ii}-C_{ij}^{ij}e_{jj}-\sum_{y>j}C_{jy}^{jj}e_{jy}\notag\\
&\quad +C_{ji}^{ji}e_{ii}+\sum_{y>i}C_{iy}^{ii}e_{iy}-\sum_{x<j}C_{xj}^{jj}e_{xj}-C_{ji}^{ji}e_{jj}.\label{L(e_ii)-L(e_jj)=[L(e_ij)e_ji]+[e_ijL(e_ji)]}
\end{align}
On the other hand, taking \cref{L(e_ii)-final} into account, we get
\begin{align}
L(e_{ii})-L(e_{jj})&=\sum_{x<i}C_{xi}^{ii}e_{xi}+\sum_{x\in X}C_{xx}^{ii}e_{xx}+\sum_{y>i}C_{iy}^{ii}e_{iy}\notag\\
&\quad -\sum_{x<j}C_{xj}^{jj}e_{xj}-\sum_{y\in X}C_{yy}^{jj}e_{yy}-\sum_{y>j}C_{jy}^{jj}e_{jy}.\label{L(e_ii)-L(e_jj)-other}
\end{align}
Comparing the right-hand sides of \cref{L(e_ii)-L(e_jj)=[L(e_ij)e_ji]+[e_ijL(e_ji)],L(e_ii)-L(e_jj)-other}, one deduces that
\begin{align}
(C_{ij}^{ij}+C_{ji}^{ji})e_{ii}-(C_{ij}^{ij}+C_{ji}^{ji})e_{jj}
&=(C_{ii}^{ii}-C_{ii}^{jj})e_{ii}+(C_{jj}^{ii}-C_{jj}^{jj})e_{jj}\notag\\
&\quad +\sum_{x\neq i,j}C_{xx}^{ii}e_{xx}-\sum_{y\neq i,j}C_{yy}^{jj}e_{yy}.\label{L(e_ii)-L(e_jj)-combined}
\end{align}
Since on the left-hand side of \cref{L(e_ii)-L(e_jj)-combined} the coefficients of $e_{ii}$ and $e_{jj}$ differ only by the sign, it follows that on the right-hand side $C_{ii}^{ii}-C_{ii}^{jj}=C_{jj}^{jj}-C_{jj}^{ii}$ for $i<j<i$. Substituting $C_{ii}^{jj}=C_{jj}^{jj}$ and $C_{jj}^{ii}=C_{ii}^{ii}$ for $i<j<i$ by \cref{C_kk^ii=C_ii^ii-for-k<i,C_ii^ii=C_ll^ii-for-i<l}, we get $C_{ii}^{ii}-C_{jj}^{jj}=C_{jj}^{jj}-C_{ii}^{ii}$. Since $\cR$ is $2$-torsion free,
\begin{align}\label{C_ii^ii=C_jj^jj}
 C_{ii}^{ii}=C_{jj}^{jj},
\end{align}
so the coefficients of $e_{ii}$ and $e_{jj}$ in \cref{L(e_ii)-L(e_jj)-combined} are zero, yielding thus \cref{C_ij^ij+C_ji^ji=0}. Finally observe that $C_{xx}^{ii}=C_{xx}^{jj}$ in \cref{L(e_ii)-L(e_jj)-combined} can be deduced from \cref{C_ii^ii=C_xx^ii,C_ii^ii=C_jj^jj}.
\end{proof}

We are now ready to prove the main result of this section.

\begin{proof}[Proof of \cref{L-proper-fin-case}]
If $|X|=1$, then $I(X,\cR)\cong \cR$, and \cref{L-proper-fin-case} is obvious. If $|X|\geq 2$, then let $L:I(X,\cR) \to I(X,\cR)$ be a Lie derivation. By \cref{L(e_ij)-gen-form,coefficients-C^ij_xy} the map $L$ has the form \cref{L(e_ii)-final,L(e_ij)-final}, where the coefficients $C^{ij}_{xy}$ satisfy \cref{C_ij^ii+C_ij^jj=0,C_ij^ij+C_jk^jk=C_ik^ik,C_ij^ij+C_ji^ji=0,C_ii^ii=C_xx^ii}.

We define an $\cR$-linear operator $D$ by
\begin{align*}
 D(e_{ii})&=\sum_{x<i}C_{xi}^{ii}e_{xi}+\sum_{y>i}C_{iy}^{ii}e_{iy},\\
 D(e_{ij})&=\sum_{x<i}C_{xi}^{ii}e_{xj}+C_{ij}^{ij}e_{ij}+\sum_{y>j}C_{jy}^{jj}e_{iy},\mbox{ if }i<j.
\end{align*}
Observe that $D(e_{ij})=\sum_{x\le y}\tilde C_{xy}^{ij}e_{xy}$, where $\tilde C_{xy}^{ij}=C_{xy}^{ij}$ for $x<y$, and $\tilde C_{xx}^{ij}=0$. Then by \cref{coefficients-C^ij_xy}
\begin{align*}
\tilde C_{ij}^{ii}+\tilde C_{ij}^{jj}&=0,\mbox{ if }i<j,\\
\tilde C^{ij}_{ij}+\tilde C^{jk}_{jk}&=\tilde C^{ik}_{ik},\mbox{ if }i\le j\le k,
\end{align*}
so \cite[Theorem 2.2]{Xiao1} makes $D$ be a derivation.

Then the linear map $F:=L-D$ satisfies
 $$
 F(e_{ij})=0,\mbox{ if }i<j,\mbox{ and }F(e_{ii})=C_{ii}^{ii}\delta,
 $$
therefore \cite[Corollary 1.3.15]{SpDo} makes $F$ be a centralizing mapping of $I(X,\cR)$.
\end{proof}

\section{The general case}\label{gen-case}

Let $(X,\le)$ be a locally finite preordered set. Denote by $\tilde{I}(X,\cR)$ the $\cR$-subspace of $I(X,\cR)$
generated by the elements $e_{xy}$ with $x\le y$, where $e_{xy}$ is defined by \cref{e_xy}. It turns out that $\tilde{I}(X,\cR)$ is a subalgebra of $I(X,\cR)$, because $e_{xy}e_{uv}=\delta_{yu}e_{xv}\in \tilde{I}(X,\cR)$. Note that $\tilde{I}(X,\cR)$ consists exactly of the functions $f\in I(X,\cR)$ which are nonzero only at a finite number of $(x,y)$. Clearly, $\tilde{I}(X,\cR)=I(X,\cR)$ if and only if $X$ is finite.

By a {\it derivation} (respectively, {\it Lie derivation}) of $\tilde{I}(X,\cR)$, we shall mean an $\cR$-map $L: \tilde{I}(X,\cR)\to I(X,\cR)$ for which $L(fg)=L(f)g+fL(g)$ (respectively, $L([f,g])=[L(f),g]+[f,L(g)]$) holds, when $f,g\in \tilde{I}(X,\cR)$. Observe that \cref{L(e_ij)-gen-form,coefficients-C^ij_xy} remain valid, when we replace a finite $X$ by a locally finite $X$, $I(X,\cR)$ by $\tilde{I}(X,\cR)$ and the ``usual'' notions of a derivation (Lie derivation) by the ones we have just introduced. Indeed, although the sums $L(e_{ij})=\sum_{x\le y}C_{xy}^{ij}e_{xy}$ are now infinite, multiplication by $e_{uv}$ on the left or on the right works as in the finite case.

We shall need the following easy formula:
\begin{align}\label{e_xfe_y}
 e_{xx}fe_{yy}=f(x,y)e_{xy}
\end{align}
for all $f\in I(X,\cR)$ and $x\le y$.

\begin{defn}\label{f_[xy]-defn}
For any $f\in I(X,\cR)$ and $x\le y$ define {\it the restriction of $f$} to $\{z\in X\mid x\le z\le y\}$ to be
\begin{align}\label{f_[xy]}
f|_x^y=\sum _{x\le u\le v\le y}f(u,v)e_{uv}.
\end{align}
Observe that the sum above is finite, so $f|_x^y\in \tilde{I}(X,\cR)$.
\end{defn}

The next lemma is straightforward.
\begin{lem}\label{prop-f_[xy]}
Let $f\in I(X,\cR)$. Then
\begin{enumerate}
 \item $x\le u\le v\le y \Rightarrow(f|_x^y)|_u^v=f|_u^v$;\label{(f_[xy])_[uv]}
 \item $(fg)(x,y)=(f|_x^yg)(x,y)=(fg|_{x}^{y})(x,y)=(f|_x^yg|_{x}^{y})(x,y)$.\label{(fg)(xy)=(f_[xy]g)(xy)}
\end{enumerate}
\end{lem}

\begin{lem}\label{f-mapsto-f_[xy]-homo}
The map $f\mapsto f|_x^y$ is a homomorphism $I(X,\cR)\to\tilde{I}(X,\cR)$.
\end{lem}

\begin{proof}
The linearity is trivial. We prove that $(fg)|_x^y=f|_x^yg|_x^y$. Suppose first that $x\le u\le v\le y$. Then, using \cref{prop-f_[xy]}, we have
\begin{align*}
(fg)|_{x}^{y}(u,v)&=(fg)(u,v)=(f|_u^vg|_u^v)(u,v)=((f|_x^y)|_u^v(g|_{x}^{y})|_u^v)(u,v)\\
&=(f|_x^yg|_{x}^{y})(u,v).
\end{align*}
Otherwise we have $x\not\le u$ or $v\not\le y$. In the first case $f|_x^y(u,w)=0$ for all $w$ by \cref{f_[xy]}. Similarly $g|_{x}^{y}(w,v)=0$ for all $w$ in the second case. According to the definition of incidence algebra in \cref{intro-prelim}, we see that in both situations $(f|_x^yg|_{x}^{y})(u,v)=0=(fg)|_{x}^{y}(u,v)$.
\end{proof}

\begin{lem}\label{Lf(xy)=Lf-restricted}
Let $L$ be a Lie derivation of $I(X,\cR)$ and $x<y$. Then
\begin{align}\label{Lf(xy)=Lf_[xy](xy)}
 L(f)(x,y)=L(f|_x^y)(x,y).
\end{align}
Moreover, if $L$ is a derivation, then \cref{Lf(xy)=Lf_[xy](xy)} holds for $x=y$ too.
\end{lem}

\begin{proof}
In view \cref{e_xfe_y}
\begin{align}
L(f)(x,y)&=[e_{xx},L(f)](x,y)\notag\\
&=\left([f,L(e_{xx})]+L([e_{xx},f])\right)(x,y)\notag\\
&=\left(fL(e_{xx})\right)(x,y)-\left(L(e_{xx})f\right)(x,y)\notag\\
&\quad +L(e_{xx}f)(x,y)-L(fe_{xx})(x,y).\label{L(f)_xy}
\end{align}
Writing $L\left([f,g]\right)=[L(f),g]+[f,L(g)]$ for the pair $e_{xx}f,e_{yy}\in I(X,\cR)$ and taking into account that $e_{xx}e_{yy}=0$, we get
\begin{align}
L(e_{xx}f)(x,y)&=\left(L(e_{xx}f)e_{yy}\right)(x,y)\notag\\
&=L\left([e_{xx}f,e_{yy}]\right)(x,y)+\left(e_{yy}L(e_{xx}f)\right)(x,y)\notag\\
&\quad -\left(e_{xx}fL(e_{yy})\right)(x,y)+\left(L(e_{yy})e_{xx}f\right)(x,y)\notag\\
&=L(e_{xx}fe_{yy})(x,y)-\left(fL(e_{yy})\right)(x,y)+L(e_{yy})(x,x)f(x,y)\notag\\
&=f(x,y)L(e_{xy})(x,y)-\left(fL(e_{yy})\right)(x,y)+L(e_{yy})(x,x)f(x,y).\label{L(e_xxf)_xy}
\end{align}
Analogously, for $e_{yy},fe_{xx}\in I(X,\cR)$, we obtain
\begin{align}
L(fe_{xx})(x,y)&=\left(L(fe_{xx})e_{yy}\right)(x,y)\notag\\
&=-L\left([e_{yy},fe_{xx}]\right)(x,y)+\left(L(e_{yy})fe_{xx}\right)(x,y)\notag\\
&\quad -\left(fe_{xx}L(e_{yy})\right)(x,y)+\left(e_{yy}L(fe_{xx})\right)(x,y)\notag\\
&=-L(e_{yy}fe_{xx})(x,y)-f(x,x)L(e_{yy})(x,y)\notag\\
&=-f(y,x)L(e_{yx})(x,y)-f(x,x)L(e_{yy})(x,y).\label{L(fe_xx)_xy}
\end{align}
It follows from \cref{L(f)_xy,L(e_xxf)_xy,L(fe_xx)_xy} that
\begin{align}
L(f)(x,y)&=\left(fL(e_{xx})\right)(x,y)-\left(L(e_{xx})f\right)(x,y)\notag\\
&\quad +f(x,y)L(e_{xy})(x,y)-\left(fL(e_{yy})\right)(x,y)+L(e_{yy})(x,x)f(x,y)\notag\\
&\quad f(x,x)L(e_{yy})(x,y)+f(y,x)L(e_{yx})(x,y).\label{L(f)(xy)}
\end{align}
If $y\not\le x$, then the last summand of \cref{L(f)(xy)} equals $0$. If $y\le x$, then $x\le z\le y\iff y\le z\le x$, so $f(y,x)=f|_{y}^{x}(y,x)=f|_x^y(y,x)$. Obviously, $f(x,y)=f|_x^y(x,y)$ and $f(x,x)=f|_x^y(x,x)$. Using \cref{(fg)(xy)=(f_[xy]g)(xy)} of \cref{prop-f_[xy]} we may replace $f$ by $f|_x^y$ in the first, second and fourth summand of the right-hand side of \cref{L(f)(xy)}. Since the same formula \cref{L(f)(xy)} is true for $f|_x^y$, we get \cref{Lf(xy)=Lf_[xy](xy)}.

Suppose now that $L$ is a derivation. Then
$$
L(e_{xx}fe_{xx})=L(e_{xx})fe_{xx}+e_{xx}L(f)e_{xx}+e_{xx}fL(e_{xx})
$$
and thus
\begin{align*}
L(f)(x,x)&=f(x,x)L(e_{xx})(x,x)-\left(L(e_{xx})f\right)(x,x)-\left(fL(e_{xx})\right)(x,x)\\
&=f|_{x}^{x}(x,x)L(e_{xx})(x,x)-\left(L(e_{xx})f|_{x}^{x}\right)(x,x)-\left(f|_{x}^{x}L(e_{xx})\right)(x,x)
\end{align*}
which coincides with $L(f|_{x}^{x})(x,x)$.
\end{proof}

\begin{rem}\label{der-is-det-by-restr}
Each derivation of $I(X,\cR)$ is fully determined by its values on the elements of $\tilde{I}(X,\cR)$, and \cref{Lf(xy)=Lf-restricted} remains valid, when we replace $I(X,\cR)$ by $\tilde{I}(X,\cR)$.
\end{rem}

\begin{lem}\label{ext-L-to-I(XR)}
Let $L:\tilde I(X,R)\to I(X,R)$ be a Lie derivation. We define
$$
\hat{L}(f)(x,y):=L(f|_x^y)(x,y)
$$
for all $f\in I(X,\cR)$, $x\le y$.
Then $\hat{L}$ is a Lie derivation of $I(X,\cR)$.
\end{lem}

\begin{proof}
Linearity of $\hat{L}$ is explained by linearity of $L$ and \cref{f-mapsto-f_[xy]-homo}. By \cref{der-is-det-by-restr} one has $\hat{L}(f)=L(f)$ for $f\in\tilde{I}(X,\cR)$. So, $\hat{L}$ is a linear extension of $L$ to the whole $I(X,\cR)$.
Given $f,g\in I(X,\cR)$ and $x\le y$, by \cref{f-mapsto-f_[xy]-homo}, \cref{(fg)(xy)=(f_[xy]g)(xy)} of \cref{prop-f_[xy]} and the fact that $L$ is a Lie derivation of $\tilde{I}(X,\cR)$
\begin{align*}
\hat{L}\left([f,g]\right)(x,y)&=L\left([f,g]|_{x}^{y}\right)(x,y)=L\left([f|_x^y,g|_{x}^{y}]\right)(x,y)\\
   &=\left([L(f|_x^y),g|_{x}^{y}]+[f|_x^y,L(g|_{x}^{y})]\right)(x,y)\\
   &=\left([L(f|_x^y),g]+[f,L(g|_{x}^{y})]\right)(x,y).
\end{align*}
Observe by \cref{der-is-det-by-restr} and \cref{(f_[xy])_[uv]} of \cref{prop-f_[xy]} that
\begin{align*}
\left(L(f|_x^y)g\right)(x,y)&=\sum_{x\le z\le y}L(f|_x^y)(x,z)g(z,y)\\
&=\sum_{x\le z\le y}L\left((f|_x^y)|_{x}^{z}\right)(x,z)g(z,y)\\
&=\sum_{x\le z\le y}L(f|_{x}^{z})(x,z)g(z,y)\\
&=\sum_{x\le z\le y}\hat{L}(f)(x,z)g(z,y)=(\hat{L}(f)g)(x,y).
\end{align*}
Similarly $\left(gL(f|_x^y)\right)(x,y)=(g\hat{L}(f))(x,y)$, so $[L(f|_x^y),g](x,y)=[\hat{L}(f),g](x,y)$. By the same argument $[f,L(g|_{x}^{y})](x,y)=[f,\hat{L}(g)](x,y)$. Thus, $\hat{L}([f,g])=[\hat{L}(f),g]+[f,\hat{L}(g)]$,
i.\,e. $\hat{L}$ is a Lie derivation of $I(X,\cR)$.
\end{proof}

\begin{rem}\label{ext-D-to-I(XR)}
Under the conditions of \cref{ext-L-to-I(XR)} if $L$ is a derivation, then $\hat{L}$ is a derivation.
\end{rem}

\begin{lem}\label{L(f)_xx=L(f)_yy}
Let $X$ be connected and $L$ be a Lie derivation of $I(X,R)$. Then $L(f)(x,x)=L(f)(y,y)$ for all $x,y\in X$.
\end{lem}

\begin{proof}
It is enough to consider the case $x<y$. We have
\begin{align}\label{L(e_xyf-fe_xy)}
L\left([e_{xy},f]\right)=L(e_{xy}f-fe_{xy})=L(e_{xy})f-fL(e_{xy})+e_{xy}L(f)-L(f)e_{xy},
\end{align}
calculating both sides at $(x,y)$, we get
\begin{align}\label{L(e_xyf-fe_xy)(xy)}
 L(e_{xy}f-fe_{xy})(x,y)=\left(L(e_{xy})f\right)(x,y)-\left(fL(e_{xy})\right)(x,y)+L(f)(y,y)-L(f)(x,x).
\end{align}
By \cref{f-mapsto-f_[xy]-homo,Lf(xy)=Lf-restricted} one has
$$
L(e_{xy}f-fe_{xy})(x,y)=L\left((e_{xy}f-fe_{xy})|_{x}^{y}\right)(x,y)=L(e_{xy}f|_x^y-f|_x^ye_{xy})(x,y).
$$
In view of \cref{(fg)(xy)=(f_[xy]g)(xy)} of \cref{prop-f_[xy]} one can replace $f$ by $f|_x^y$ in the first two summands of the right-hand side of \cref{L(e_xyf-fe_xy)(xy)}. Thus,
\begin{align*}
L(e_{xy}f|_x^y-f|_x^ye_{xy})(x,y)&=\left(L(e_{xy})f|_x^y)(x,y)-(f|_x^yL(e_{xy})\right)(x,y)\\
&\quad +L(f)(y,y)-L(f)(x,x).
\end{align*}
On the other hand, writing \cref{L(e_xyf-fe_xy)} for $f|_x^y$ and evaluating at $(x,y)$, we obtain
\begin{align*}
L(e_{xy}f|_x^y-f|_x^ye_{xy})(x,y)&=\left(L(e_{xy})f|_x^y)(x,y)-(f|_x^yL(e_{xy})\right)(x,y)\\
&\quad +L(f|_x^y)(y,y)-L(f|_x^y)(x,x).
\end{align*}
It follows that
$$
L(f)(y,y)-L(f)(x,x)=L(f|_x^y)(y,y)-L(f|_x^y)(x,x),
$$
the latter being zero by \cref{coefficients-C^ij_xy}.
\end{proof}

\begin{defn}\label{diag-defn}
Given $f\in I(X,\cR)$, we define the {\it diagonal} of $f$ by
$$
f_d(x,y)=
\begin{cases}
f(x,y), & x=y,\\
0, & x\neq y.
\end{cases}
$$
\end{defn}

\begin{thrm}\label{L=D+R}
Let $X$ be connected and $\cR$ be $2$-torsion free. Then any Lie derivation of $I(X,\cR)$ is the sum of a derivation of $I(X,\cR)$ and a linear map from $I(X,\cR)$ to its center.
\end{thrm}

\begin{proof}
Let $L$ be a Lie derivation of $I(X,\cR)$. Define $R(f)=L(f)_d$ and $D(f)=L(f)-R(f)$. Then $R$ is a linear map from $I(X,\cR)$ to the center of $I(X,\cR)$ by \cref{L(f)_xx=L(f)_yy}. It remains to prove that $D$ is a derivation of $I(X,\cR)$.
For any $f\in\tilde{I}(X,\cR)$ one has $D(f)=L(f)-L(f)_d$, which is a derivation of $\tilde{I}(X,R)$ by the proof of \cref{L-proper-fin-case}. Therefore, $D$ extends to a derivation $\hat{D}$ of $I(X,\cR)$ by \cref{ext-D-to-I(XR)}. Observe that
\begin{align}\label{hat D(f)(xy)}
 \hat{D}(f)(x,y)=D(f|_x^y)(x,y)=L(f|_x^y)(x,y)-L(f|_x^y)_d(x,y).
\end{align}
If $x<y$, then the latter coincides with $L(f|_x^y)(x,y)$, which is $L(f)(x,y)$ by \cref{Lf(xy)=Lf-restricted}. But $L(f)_d(x,y)=0$ in this case, so $L(f)(x,y)=D(f)(x,y)$. And if $x=y$, then the right-hand side of \cref{hat D(f)(xy)} is zero, which coincides with $L(f)(x,x)-L(f)_d(x,x)=D(f)(x,x)$. Thus, $\hat{D}=D$, so $D$ is a derivation of $I(X,\cR)$.
 \end{proof}

\section*{Acknowledgements}
The authors thank professor Zhankui Xiao who prompted them to focus their studies on Lie derivations.

\bibliography{bibl}{}

\begin{thebibliography}{10}

\bibitem{Ben-LieDer-TM}
{\sc Benkovi\v{c}, D.}
\newblock {Lie derivations on triangular matrices}.
\newblock {\em Linear Multilinear Algebra 55}, 6 (2007), 619--626.

\bibitem{Ben-GenLieDer-TA}
{\sc Benkovi\v{c}, D.}
\newblock {Generalized {L}ie derivations on triangular algebras}.
\newblock {\em Linear Algebra Appl. 434}, 6 (2011), 1532--1544.

\bibitem{Bre93}
{\sc Bre\v{s}ar, M.}
\newblock {Commuting traces of biadditive mappings, commutativity-preserving
  mappings and {L}ie mappings}.
\newblock {\em Trans. Amer. Math. Soc. 335}, 2 (1993), 525--546.

\bibitem{Bre04}
{\sc Bre\v{s}ar, M.}
\newblock {Commuting maps: a survey}.
\newblock {\em Taiwanese J. Math. 8}, 3 (2004), 361--397.

\bibitem{BruFS}
{\sc {Brusamarello}, R., {Fornaroli}, {\'E}.~Z., and {Santulo Jr}, E.~A.}
\newblock {Classification of involutions on finitary incidence algebras}.
\newblock {\em {Int. J. Algebra Comput.} 24}, 8 (2014), 1085--1098.

\bibitem{BruL}
{\sc Brusamarello, R., and Lewis, D.~W.}
\newblock {Automorphisms and involutions on incidence algebras}.
\newblock {\em Linear Multilinear Algebra 59}, 11 (2011), 1247--1267.

\bibitem{Cheung}
{\sc {Cheung}, W.~S.}
\newblock {Lie derivations of triangular algebras.}
\newblock {\em {Linear Multilinear Algebra} 51}, 3 (2003), 299--310.

\bibitem{DuWang-LieDer-GMA}
{\sc Du, Y., and Wang, Y.}
\newblock {Lie derivations of generalized matrix algebras}.
\newblock {\em Linear Algebra Appl. 437}, 11 (2012), 2719--2726.

\bibitem{Her}
{\sc {Herstein}, I.~N.}
\newblock {Lie and Jordan structures in simple associative rings}.
\newblock {\em {Bull. Am. Math. Soc.} 67\/} (1961), 517--531.

\bibitem{Kh-aut}
{\sc Khripchenko, N.~S.}
\newblock {Automorphisms of finitary incidence rings}.
\newblock {\em Algebra and Discrete Math. 9}, 2 (2010), 78--97.

\bibitem{Kh-der}
{\sc Khripchenko, N.~S.}
\newblock {Derivations of finitary incidence rings}.
\newblock {\em Comm. Algebra 40}, 7 (2012), 2503--2522.

\bibitem{Kh}
{\sc Khrypchenko, M.}
\newblock {Jordan derivations of finitary incidence rings}.
\newblock {\em Linear Multilinear Algebra 64}, 10 (2016), 2104--2118.

\bibitem{Kopp}
{\sc {Koppinen}, M.}
\newblock {Automorphisms and higher derivations of incidence algebras}.
\newblock {\em {J. Algebra} 174}, 2 (1995), 698--723.

\bibitem{Lu}
{\sc {Lu}, F.}
\newblock {Lie derivations of certain CSL algebras}.
\newblock {\em {Isr. J. Math.} 155\/} (2006), 149--156.

\bibitem{Mar64}
{\sc {Martindale 3rd}, W.~S.}
\newblock {Lie derivations of primitive rings}.
\newblock {\em {Mich. Math. J.} 11\/} (1964), 183--187.

\bibitem{MaV}
{\sc {Mathieu}, M., and {Villena}, A.~R.}
\newblock {The structure of Lie derivations on $C^{\ast}$-algebras}.
\newblock {\em {J. Funct. Anal.} 202}, 2 (2003), 504--525.

\bibitem{Miers}
{\sc {Miers}, C.}
\newblock {Lie derivations of von Neumann algebras}.
\newblock {\em {Duke Math. J.} 40\/} (1973), 403--409.

\bibitem{Sp}
{\sc {Spiegel}, E.}
\newblock {On the automorphisms of incidence algebras}.
\newblock {\em {J. Algebra} 239}, 2 (2001), 615--623.

\bibitem{SpDo}
{\sc {Spiegel}, E., and {O'Donnell}, C.~J.}
\newblock {\em {Incidence algebras}}.
\newblock New York, NY: Marcel Dekker, 1997.

\bibitem{St}
{\sc {Stanley}, R.}
\newblock {Structure of incidence algebras and their automorphism groups}.
\newblock {\em {Bull. Am. Math. Soc.} 76\/} (1970), 1236--1239.

\bibitem{Stanley}
{\sc Stanley, R.}
\newblock {\em {Enumerative Combinatorics}}, vol.~1.
\newblock Cambridge University Press, 1997.

\bibitem{Wa}
{\sc {Ward}, M.}
\newblock {Arithmetic functions on rings}.
\newblock {\em {Ann. Math. (2)} 38\/} (1937), 725--732.

\bibitem{Xiao1}
{\sc Xiao, Z.~K.}
\newblock {Jordan derivations of incidence algebras}.
\newblock {\em Rocky Mountain J. Math. 45}, 4 (2015), 1357--1368.

\bibitem{XiaoWei-LieTriple-TA}
{\sc Xiao, Z.~K., and Wei, F.}
\newblock {Lie triple derivations of triangular algebras}.
\newblock {\em Linear Algebra Appl. 437}, 5 (2012), 1234--1249.

\end{thebibliography}
\bibliographystyle{acm}

\end{document}